\documentclass[11pt,a4paper]{amsart}
\usepackage[centering, margin=1in]{geometry}
\usepackage{amsmath,amsfonts,amsthm,amssymb,enumitem,mathrsfs,mathtools}
\usepackage[stretch=10]{microtype}
\usepackage[unicode]{hyperref}
\usepackage{xcolor}
\definecolor{dark-red}{rgb}{0.4,0.15,0.15}
\definecolor{dark-blue}{rgb}{0.15,0.15,0.4}
\definecolor{medium-blue}{rgb}{0,0,0.5}
\hypersetup{
	pdftitle={Quantitative Equidistribution on Hyperbolic Surfaces and Arithmetic Applications},
	pdfauthor={Peter Humphries},
	pdfnewwindow=true,
	colorlinks, linkcolor={dark-red},
	citecolor={dark-blue}, urlcolor={medium-blue}
}
\allowdisplaybreaks

\newcommand{\af}{\mathfrak{a}}
\newcommand{\BB}{\mathcal{B}}
\newcommand{\C}{\mathbb{C}}
\newcommand{\CC}{\mathcal{C}}
\newcommand{\dee}{\partial}
\newcommand{\e}{\varepsilon}

\newcommand{\FF}{\mathcal{F}}

\newcommand{\Hb}{\mathbb{H}}

\newcommand{\Q}{\mathbb{Q}}
\newcommand{\R}{\mathbb{R}}

\newcommand{\T}{\mathbb{T}}

\newcommand{\Wscr}{\mathscr{W}}
\newcommand{\Z}{\mathbb{Z}}

\DeclareMathOperator{\ad}{ad}

\DeclareMathOperator{\arsinh}{arsinh}

\DeclareMathOperator{\GL}{GL}
\DeclareMathOperator{\Ht}{ht}
\DeclareMathOperator{\Lip}{Lip}

\DeclareMathOperator{\sgn}{sgn}
\DeclareMathOperator{\SL}{SL}

\numberwithin{equation}{section}
\newtheorem{theorem}[equation]{Theorem}

\newtheorem{lemma}[equation]{Lemma}

\theoremstyle{remark}
\newtheorem{remark}[equation]{Remark}
\theoremstyle{definition}
\newtheorem{definition}[equation]{Definition}
\newtheorem{example}[equation]{Example}
\begin{document}

\title{Quantitative Equidistribution on Hyperbolic Surfaces and Arithmetic Applications}

\author{Peter Humphries}

\address{Department of Mathematics, University of Virginia, Charlottesville, VA 22904, USA}

\email{\href{mailto:pclhumphries@gmail.com}{pclhumphries@gmail.com}}

\urladdr{\href{https://sites.google.com/view/peterhumphries/}{https://sites.google.com/view/peterhumphries/}}

\thanks{The author was supported by the National Science Foundation (grant DMS-2302079) and by the Simons Foundation (award 965056).}

\begin{abstract}
The Wasserstein distance quantifies the distance between two probability measures on a metric space. We prove an analogue of the Berry--Esseen inequality for the Wasserstein distance on a finite area hyperbolic surface. This inequality controls the Wasserstein distance via an average of Weyl sums, which are integrals of Maa\ss{} cusp forms and Eisenstein series with respect to these probability measures. As applications, we prove upper bounds for the Wasserstein distance for some equidistribution problems on the modular surface $\SL_2(\Z) \backslash \Hb$, namely Duke's theorems on the equidistribution of Heegner points and of closed geodesics and Watson's theorem on the mass equidistribution of Hecke--Maa\ss{} cusp forms conditionally under the assumption of the generalised Lindel\"{o}f hypothesis.
\end{abstract}

\maketitle

\section{Introduction}

\subsection{The Wasserstein Distance and the Berry--Esseen Inequality}

Let $(X,\rho)$ be a metric space. A sequence of Borel probability measures $(\mu_k)$ on $X$ is said to \emph{equidistribute} on $X$ with respect to a limiting Borel probability measure $\mu$ if
\[\lim_{k \to \infty} \int_{X} f(x) \, d\mu_k(x) = \int_{X} f(x) \, d\mu(x)\]
for every continuous bounded function $f : X \to \C$. By the Portmanteau theorem, this is equivalent to the statement that $\lim_{k \to \infty} \mu_k(B) = \mu(B)$ for every $\mu$-continuity set $B \subseteq X$ (namely a Borel set whose boundary has $\mu$-measure zero). Similarly, it is also equivalent to the statement that
\[\lim_{k \to \infty} \int_{X} f(x) \, d\mu_k(x) = \int_{X} f(x) \, d\mu(x)\]
for every bounded Lipschitz function $f : X \to \R$, where we recall that $f$ is an $L$-Lipschitz function for some $L \geq 0$ if $\left|f(x) - f(y)\right| \leq L \rho(x,y)$ for all $x,y \in X$.

To quantify the rate of equidistribution is to give a measure of the distance between $\mu_k$ and $\mu$. One such quantification of the distance between two probability measures is the $1$-Wasserstein distance. Given two Borel probability measures $\nu_1,\nu_2$ on a Polish space $(X,\rho)$ (namely a metric space that is complete and separable), the $1$-Wasserstein distance between $\nu_1$ and $\nu_2$ is
\[\Wscr_1(\nu_1,\nu_2) \coloneqq \inf_{\pi \in \Pi(\nu_1,\nu_2)} \int_{X \times X} \rho(x,y) \, d\pi(x,y),\]
where $\Pi(\nu_1,\nu_2)$ denotes the set of Borel probability measures $\pi$ on $X \times X$ with marginals $\nu_1$ and $\nu_2$, so that $\pi(B \times X) = \nu_1(B)$ and $\pi(X \times B) = \nu_2(B)$ for every Borel set $B \subseteq X$. Informally, this measures the cost of moving from the measure $\nu_1$ to the measure $\nu_2$. The $1$-Wasserstein distance is of central importance in optimal transport; see, for example, \cite{Vil03}. Moreover, the $1$-Wasserstein distance defines a metric on the space of all Borel probability measures $\nu$ on $X$ for which $\int_{X} \rho(x_0,x) \, d\nu(x)$ is finite for some (and hence for all) $x_0 \in X$. The convergence of a sequence of measures $(\mu_k)$ to a limiting measure $\mu$ with respect to this metric is simply equidistribution.

The definition of the $1$-Wasserstein distance is intrinsic and satisfies various invariance properties and natural inequalities; see, for example, \cite[Theorem 1.2]{KU25} for several such properties. In general, however, the $1$-Wasserstein distance is not easily estimated except in special situations. When $X = \R$ and $\rho(x,y) = |x - y|$, one has the simpler formulation of the $1$-Wasserstein distance in terms of the cumulative distribution functions of $\nu_1$ and $\nu_2$, namely
\[\Wscr_1(\nu_1,\nu_2) = \int_{-\infty}^{\infty} |\nu_1((-\infty,x]) - \nu_2((-\infty,x])| \, dx.\]
The Berry--Esseen inequality then bounds this quantity in terms of the Fourier transforms $\widehat{\nu_j}(t) \coloneqq \int_{-\infty}^{\infty} e^{-2\pi itx} \, d\nu_j(x)$.

\begin{theorem}[{Berry--Esseen inequality \cite[Corollary 8.3]{Bob16}}]
Let $\nu_1,\nu_2$ be Borel probability measures on $\R$. For $T \geq 1$, we have that
\begin{multline*}
\Wscr_1(\nu_1,\nu_2)	\\
\leq \frac{16\sqrt{2}}{\sqrt{3}\pi T} + \left(\frac{1}{2\pi} \int_{-T}^{T} \left|\frac{\widehat{\nu_1}(t) - \widehat{\nu_2}(t)}{t}\right|^2 \, dt\right)^{1/2} + \left(\frac{1}{(2\pi)^3} \int_{-T}^{T} \left|\frac{d}{dt} \frac{\widehat{\nu_1}(t) - \widehat{\nu_2}(t)}{t}\right|^2 \, dt\right)^{1/2}.
\end{multline*}
\end{theorem}

Recently, an inequality of this form for the $1$-Wasserstein distance was extended to the setting of the $n$-torus $\T^n = (\R/\Z)^n$ by Bobkov and Ledoux and independently by Borda.

\begin{theorem}[{Bobkov--Ledoux \cite[Proposition 2]{BL21}, Borda \cite[Proposition 3]{Bor21a} (see also \cite[Theorem 1.2 (7)]{KU25})}]
\label{thm:BobkovLedoux}
Let $\nu_1$ and $\nu_2$ be Borel probability measures on $\T^n$ and let $\widehat{\nu_j}(m) \coloneqq \int_{\T^n} e^{-2\pi i m \cdot x} \, d\nu_j(x)$ denote the $m$-th Fourier coefficient of $\nu_j$. For $T \geq 1$, we have that
\[\Wscr_1(\nu_1,\nu_2) \leq \frac{4\sqrt{3n}}{T} + \left(\sum_{\substack{m = (m_1,\ldots,m_n) \in \Z^n \setminus \{0\} \\ |m_1|,\ldots,|m_n| \leq T}} \frac{\left|\widehat{\nu_1}(m) - \widehat{\nu_2}(m)\right|^2}{|m|^2}\right)^{1/2}.\]
\end{theorem}

Borda has also proven a similar bound on compact connected Lie groups \cite[Theorem 1]{Bor21b} (see also \cite[Theorem 4.1]{KU25} and more generally \cite[Theorem 4]{BC25}).

Kowalski and Untrau \cite{KU25} recently investigated several equidistribution problems in analytic number theory that are related to exponential sums over finite fields. They gave effective bounds for the $1$-Wasserstein distance for these equidistribution problems by first applying the Berry--Esseen inequality in the relevant setting, then showing that the Fourier coefficients of the measures that appear in this inequality are related to exponential sums, and finally inputting pre-existing bounds for such exponential sums. See also \cite{Gra20,Stei21} for results on the quantification of the equidistribution of quadratic residues in terms of the $1$-Wasserstein distance.

In this paper, we investigate the $1$-Wasserstein distance in another setting relevant to problems in analytic number theory, namely equidistribution on finite area hyperbolic surfaces; see \cite[Chapters 1 and 2]{Iwa02} for relevant background information on such surfaces. We prove forms of the Berry--Esseen inequality in this setting, which are stated in \hyperref[thm:compact]{Theorems \ref*{thm:compact}} and \ref{thm:noncompact}. We then apply this inequality to prove bounds for the $1$-Wasserstein distance in several arithmetic equidistribution problems on finite area hyperbolic surfaces, which we state in \hyperref[sec:arithmetic]{Section \ref*{sec:arithmetic}}.

\subsection{A Berry--Esseen Inequality for Finite Area Hyperbolic Surfaces}

Let $\Hb \coloneqq \{z = x + iy \in \C : x \in \R, \ y \in \R_+\}$ denote the upper half-plane. The group $\SL_2(\R)$ acts on $\Hb$ via M\"{o}bius transformations, namely $gz \coloneqq \frac{az + b}{cz + d}$ for $z \in \Hb$ and $g = \begin{psmallmatrix} a & b \\ c & d \end{psmallmatrix} \in \SL_2(\R)$. The upper half-plane is a Riemannian manifold with the metric derived from the Poincar\'{e} differential $ds^2 = y^{-2} \, dx^2 + y^{-2} \, dy^2$ and associated area form $d\mu(z) = y^{-2} \, dx \, dy$. The distance function is given by
\begin{equation}
\label{eqn:rhodefeq}
\rho(z,w) = \log \frac{|z - \overline{w}| + |z - w|}{|z - \overline{w}| - |z - w|} = 2 \arsinh \frac{|z - w|}{2 \sqrt{\Im(z) \Im(w)}}.
\end{equation}
We additionally let
\begin{equation}
\label{eqn:udefeq}
u(z,w) \coloneqq \frac{|z - w|^2}{4 \Im(z) \Im(w)} = \sinh^2 \frac{\rho(z,w)}{2}.
\end{equation}
The distance function $\rho$ (and hence also the associated function $u$) and the area form $\mu$ are $\SL_2(\R)$-invariant, in the sense that $\rho(gz,gw) = \rho(z,w)$, $u(gz,gw) = u(z,w)$, and $d\mu(gz) = d\mu(z)$ for $z,w \in \Hb$ and $g \in \SL_2(\R)$ (see \cite[Chapter 1]{Iwa02}).

Let $\Gamma \subset \SL_2(\R)$ be a lattice, so that the quotient space $\Gamma \backslash \Hb$ has finite area with respect to $\mu$. The distance function $\rho$ on $\Hb$ descends to the distance function on $\Gamma \backslash \Hb$ given by
\[\rho_{\Gamma \backslash \Hb}(z,w) \coloneqq \min_{\gamma \in \Gamma} \rho(z,\gamma w).\]
We are interested in quantifying the distance between two probability measures on $\Gamma \backslash \Hb$. Given two Borel probability measures $\nu_1,\nu_2$ on $\Gamma \backslash \Hb$, the $1$-Wasserstein distance between $\nu_1$ and $\nu_2$ is
\[\Wscr_1(\nu_1,\nu_2) \coloneqq \inf_{\pi \in \Pi(\nu_1,\nu_2)} \int_{\Gamma \backslash \Hb \times \Gamma \backslash \Hb} \rho_{\Gamma \backslash \Hb}(z,w) \, d\pi(z,w),\]
where $\Pi(\nu_1,\nu_2)$ denotes the set of Borel probability measures $\pi$ on $\Gamma \backslash \Hb \times \Gamma \backslash \Hb$ with marginals $\nu_1$ and $\nu_2$, so that $\pi(B \times \Gamma \backslash \Hb) = \nu_1(B)$ and $\pi(\Gamma \backslash \Hb \times B) = \nu_2(B)$ for every Borel set $B \subseteq \Gamma \backslash \Hb$.

There is a dual formulation of the $1$-Wasserstein distance in terms of Lipschitz functions. An $L$-Lipschitz function on $\Gamma \backslash \Hb$ for some $L \geq 0$ is a function $F : \Hb \to \R$ satisfying $F(\gamma z) = F(z)$ for all $\gamma \in \Gamma$ and $z \in \Hb$ and
\[|F(z) - F(w)| \leq L \rho(z,w)\]
for all $z,w \in \Hb$; equivalently, we may view $F$ as a function on $\Gamma \backslash \Hb$ that satisfies
\[|F(z) - F(w)| \leq L \rho_{\Gamma \backslash \Hb}(z,w)\]
for all $z,w \in \Gamma \backslash \Hb$. Let $\Lip_1(\Gamma \backslash \Hb)$ denote the space of all $1$-Lipschitz functions on $\Gamma \backslash \Hb$. Via the Kantorovich--Rubinstein duality theorem \cite[Theorem 1.3]{Vil03}, we have that
\begin{equation}
\label{eqn:KRduality}
\Wscr_1(\nu_1,\nu_2) = \sup_{F \in \Lip_1(\Gamma \backslash \Hb)} \left|\int_{\Gamma \backslash \Hb} F(z) \, d\nu_1(z) - \int_{\Gamma \backslash \Hb} F(z) \, d\nu_2(z)\right|.
\end{equation}
This dual formulation turns out to be more useful for our purposes.

\subsubsection{The Cocompact Case}

We recall that when $\Gamma$ is cocompact, $L^2(\Gamma \backslash \Hb)$ has an orthonormal basis consisting of the constant function $\mu(\Gamma \backslash \Hb)^{-1/2}$ and of a countably infinite collection $\BB$ of nonconstant Maa\ss{} cusp forms (see \cite[Chapter 4]{Iwa02}). Each Maa\ss{} cusp form $f$ is $L^2$-normalised, so that $\int_{\Gamma \backslash \Hb} |f(z)|^2 \, d\mu(z) = 1$, and is a Laplacian eigenfunction with Laplacian eigenvalue $\lambda_f = \frac{1}{4} + t_f^2$, where $t_f \in \R \cup i(-\frac{1}{2},\frac{1}{2})$ denotes the spectral parameter of $f$. We shall control the size of $\Wscr_1(\nu_1,\nu_2)$ in terms of a weighted average of the Weyl sums $\int_{\Gamma \backslash \Hb} f(z) \, d\nu_j(z)$.

\begin{theorem}
\label{thm:compact}
Let $\Gamma$ be a cocompact lattice in $\SL_2(\R)$. Let $\nu_1,\nu_2$ be Borel probability measures on $\Gamma \backslash \Hb$. Then for all $T \geq 1$,
\begin{equation}
\label{eqn:compact}
\Wscr_1(\nu_1,\nu_2) \ll \frac{1}{T} + \mu(\Gamma \backslash \Hb)^{\frac{1}{2}} \left(\sum_{f \in \BB} \frac{e^{-\frac{t_f^2}{T^2}}}{\frac{1}{4} + t_f^2} \left|\int_{\Gamma \backslash \Hb} f(z) \, d\nu_1(z) - \int_{\Gamma \backslash \Hb} f(z) \, d\nu_2(z)\right|^2 \right)^{\frac{1}{2}},
\end{equation}
where the implicit constant is independent of $\Gamma$.
\end{theorem}

\begin{remark}
Note that $f \in \BB$ is $L^2$-normalised with respect to the measure $\mu$ on $\Gamma \backslash \Hb$, which in general need not be a probability measure. The presence of the normalisation factor $\mu(\Gamma \backslash \Hb)^{1/2}$ in the second term on the right-hand side of \eqref{eqn:compact} is therefore natural since the rescaled cusp forms $\mu(\Gamma \backslash \Hb)^{1/2} f$ are $L^2$-normalised with respect the probability measure $\mu(\Gamma \backslash \Hb)^{-1} \, \mu$ on $\Gamma \backslash \Hb$.
\end{remark}

\begin{remark}
In \hyperref[thm:BobkovLedoux]{Theorem \ref*{thm:BobkovLedoux}}, there is a sharp cut-off, namely a sum over $n$-tuples of integers $(m_1,\ldots,m_n) \in \Z^n \setminus \{0\}$ for which $|m_1|,\ldots,|m_n| \leq T$. In \hyperref[thm:compact]{Theorem \ref*{thm:compact}}, on the other hand, we instead have a smoothed cut-off, namely a sum over Maa\ss{} cusp forms $f \in \BB$ weighted by $e^{-t_f^2/T^2}$. We explain the need for such a smoothing in \hyperref[rem:cutoff]{Remark \ref*{rem:cutoff}}.
\end{remark}

\subsubsection{The Noncocompact Case}
\label{sec:noncompact}

When $\Gamma$ is noncocompact, $\Gamma \backslash \Hb$ has a finite yet nonempty collection of cusps $\af$. Associated to each cusp is an Eisenstein series $E_{\af}(z,s)$. The spectral decomposition of $L^2(\Gamma \backslash \Hb)$ then consists of a discrete spectrum consisting of the constant function $\mu(\Gamma \backslash \Hb)^{-1/2}$ and a countable collection $\BB$ of Maa\ss{} cusp forms as well as a continuous spectrum spanned by Eisenstein series $E_{\af}(z,\frac{1}{2} + it)$, where $t \in \R$, as $\af$ runs over the cusps of $\Gamma \backslash \Hb$ (see \cite[Chapters 4 and 7]{Iwa02}).

In general, Eisenstein series need not be integrable with respect to a given Borel probability measure $\nu$ on $\Gamma \backslash \Hb$ due to the fact that $E_{\af}(z,\frac{1}{2} + it)$ is unbounded as a function of $z$. To rectify this, we impose additional conditions on $\nu$. We recall that for each cusp $\af$ of $\Gamma \backslash \Hb$, there exists a scaling matrix $\sigma_{\af} \in \SL_2(\R)$ for which $\sigma_{\af} \infty = \af$ such that $\sigma_{\af} \begin{psmallmatrix} 1 & 1 \\ 0 & 1 \end{psmallmatrix} \sigma_{\af}^{-1}$ and $\begin{psmallmatrix} -1 & 0 \\ 0 & -1 \end{psmallmatrix}$ together generate the stabiliser $\Gamma_{\af} \coloneqq \{\gamma \in \Gamma : \gamma \af = \af\}$ of $\af$ with respect to $\Gamma$. For $Y \geq 1$, the \emph{cuspidal zone} is the set
\[\FF_{\af}(Y) \coloneqq \{z \in \Hb : 0 < \Re(\sigma_{\af}^{-1} z) < 1, \ \Im(\sigma_{\af}^{-1} z) > Y\}\]
(see \cite[Section 2.2]{Iwa02}).

\begin{definition}
Let $\Gamma$ be a noncocompact lattice in $\SL_2(\R)$. A finite Borel measure $\nu$ on $\Gamma \backslash \Hb$ is \emph{$Y^{-\alpha}$-cuspidally tight} for some $\alpha \geq 0$ if for every cuspidal zone $\FF_{\af}(Y)$ of $\Gamma \backslash \Hb$, we have that $\nu(\FF_{\af}(Y)) \ll_{\Gamma} Y^{-\alpha}$ as $Y$ tends to infinity.
\end{definition}

\begin{example}
The measure $d\mu(z) = y^{-2} \, dx \, dy$ on $\Gamma \backslash \Hb$ is $Y^{-\alpha}$-cuspidally tight for any $\alpha \leq 1$.
\end{example}

The $Y^{-\alpha}$-cuspidally tightness of a finite Borel measure $\nu$ ensures that $\int_{\Gamma \backslash \Hb} \Ht_{\Gamma}(z)^{\beta} \, d\nu(z)$ is finite for all $\beta \in [0,\alpha)$, where $\Ht_{\Gamma}(z) \coloneqq \max_{\gamma \in \Gamma} \Im(\gamma z)$. In particular, if $\nu$ is $Y^{-\alpha}$-cuspidally tight for some $\alpha > 0$, then for every $L \geq 0 $, an $L$-Lipschitz function $F : \Gamma \backslash \Hb \to \C$ is $\nu$-integrable. To see this, observe that for any fixed $w \in \Hb$,
\begin{align*}
\int_{\Gamma \backslash \Hb} |F(z)| \, d\nu(z) & \leq \int_{\Gamma \backslash \Hb} |F(z) - F(w)| \, d\nu(z) + |F(w)| \nu(\Gamma \backslash \Hb)	\\
& \leq L \int_{\Gamma \backslash \Hb} \rho(z,w) \, d\nu(z) + |F(w)| \nu(\Gamma \backslash \Hb),
\end{align*}
and recall from \eqref{eqn:rhodefeq} that $\rho(z,w)$ grows at most logarithmically as a function of $\Im(z)$. Moreover, if $\nu$ is $Y^{-\alpha}$-cuspidally tight for some $\alpha > \frac{1}{2}$, then every Eisenstein series $E_{\af}(z,\frac{1}{2} + it)$ is $\nu$-integrable.

With this definition in hand, we may now state the analogue of \hyperref[thm:compact]{Theorem \ref*{thm:compact}} in the setting of noncocompact lattices.

\begin{theorem}
\label{thm:noncompact}
Let $\Gamma$ be a noncocompact lattice in $\SL_2(\R)$. Let $\nu_1,\nu_2$ be Borel probability measures on $\Gamma \backslash \Hb$ that are both $Y^{-1/2 - \delta}$-cuspidally tight for some $\delta > 0$. Then for all $T \geq 1$,
\begin{multline}
\label{eqn:noncompact}
\Wscr_1(\nu_1,\nu_2) \ll \frac{1}{T} + \mu(\Gamma \backslash \Hb)^{\frac{1}{2}} \left(\sum_{f \in \BB} \frac{e^{-\frac{t_f^2}{T^2}}}{\frac{1}{4} + t_f^2} \left|\int_{\Gamma \backslash \Hb} f(z) \, d\nu_1(z) - \int_{\Gamma \backslash \Hb} f(z) \, d\nu_2(z)\right|^2 \right.	\\
\left. + \sum_{\af} \frac{1}{4\pi} \int_{-\infty}^{\infty} \frac{e^{-\frac{t^2}{T^2}}}{\frac{1}{4} + t^2} \left|\int_{\Gamma \backslash \Hb} E_{\af}\left(z,\frac{1}{2} + it\right) \, d\nu_1(z) - \int_{\Gamma \backslash \Hb} E_{\af}\left(z,\frac{1}{2} + it\right) \, d\nu_2(z)\right|^2 \, dt\right)^{\frac{1}{2}},
\end{multline}
where the implicit constant is independent of $\Gamma$ and $\delta$.
\end{theorem}

\subsection{Arithmetic Applications}
\label{sec:arithmetic}

We now state some arithmetic applications of \hyperref[thm:compact]{Theorems \ref*{thm:compact}} and \ref{thm:noncompact}. We take $\Gamma$ to be the modular group $\SL_2(\Z)$, so that $\Gamma \backslash \Hb$ is of finite area but is noncompact and has a single cusp, namely $\af = \infty$. We let
\[\nu \coloneqq \frac{1}{\mu(\Gamma \backslash \Hb)} \mu\]
denote the probability Haar measure on $\Gamma \backslash \Hb$.\footnote{Our results below also hold more generally when $\Gamma$ is the Hecke congruence group $\Gamma_0(q)$ consisting of matrices in $\SL_2(\Z)$ whose lower left entry is a multiple of a fixed positive integer $q$. Similarly, they also hold when $\Gamma$ is a cocompact lattice arising as the image of the group of norm one units of an Eichler order in a quaternion division algebra. The crucial input behind both of these generalisations is the fact that both \hyperref[lem:DukeWaldspurger]{Lemmata \ref*{lem:DukeWaldspurger}} and \ref{lem:WatsonIchino} below extend to these more general settings.}

\subsubsection{Duke's Theorem on the Equidistribution of Heegner Points and of Closed Geodesics}

Let $D < 0$ be a fundamental discriminant. Each ideal class in the class group of the imaginary quadratic field $\Q(\sqrt{D})$ is associated to a $\Gamma$-orbit of primitive irreducible integral binary quadratic forms $Q(x,y) = ax^2 + bxy + cy^2$ of discriminant $b^2 - 4ac = D$. In turn, such a $\Gamma$-orbit is associated to a $\Gamma$-orbit of points $(-b + \sqrt{D})/2a$ in the upper half-plane $\Hb$, or equivalently a single Heegner point on the modular surface $\Gamma \backslash \Hb$. We denote by $\Lambda_D$ the set of Heegner points of discriminant $D$ on $\Gamma \backslash \Hb$ and we define the Borel probability measure $\nu_D$ given on Borel sets $B \subseteq \Gamma \backslash \Hb$ by
\begin{equation}
\label{eqn:nuDneq}
\nu_D(B) \coloneqq \frac{\#\ (\Lambda_D \cap B)}{\# \Lambda_D}.
\end{equation}

Similarly, let $D > 0$ be a positive fundamental discriminant. Each narrow ideal class in the narrow class group of the real quadratic field $\Q(\sqrt{D})$ is associated to a $\Gamma$-orbit of primitive irreducible integral binary quadratic forms $Q(x,y) = ax^2 + bxy + cy^2$ of discriminant $b^2 - 4ac = D$. In turn, such a $\Gamma$-orbit is associated to a $\Gamma$-orbit of oriented complete geodesics in the upper half-plane that intersect the real line at $(-b \pm \sqrt{D})/2a$, or equivalently a single oriented closed geodesic $\CC \subset \Gamma \backslash \Hb$. We again let $\Lambda_D$ denote the set of oriented closed geodesics of discriminant $D$ on $\Gamma \backslash \Hb$ and we define the Borel probability measure $\nu_D$ via
\begin{equation}
\label{eqn:nuDpos}
\nu_D(B) \coloneqq \frac{\sum_{\CC \in \Lambda_D} \ell(\CC \cap B)}{\sum_{\CC \in \Lambda_D} \ell(\CC)},
\end{equation}
where $\ell(\CC) \coloneqq \int_{\CC} \, ds$ denotes the hyperbolic length of a curve $\CC$ with respect to the measure $ds^2 := y^{-2} \, dx^2 + y^{-2} \, dy^2$.\

Duke proved the following equidistribution theorem for these measures.

\begin{theorem}[{Duke \cite{Duk88} (see also \cite{ELMV12})}]
As $D$ tends to negative infinity along negative fundamental discriminants, the sequence of Borel probability measures $\nu_D$ associated to Heegner points of discriminant $D$ via \eqref{eqn:nuDneq} equidistributes on the modular surface with respect to the probability Haar measure $\nu$.

Similarly, as $D$ tends to infinity along positive fundamental discriminants, the sequence of Borel probability measures $\nu_D$ associated to oriented closed geodesics of discriminant $D$ via \eqref{eqn:nuDpos} equidistributes on the modular surface with respect to the probability Haar measure $\nu$.
\end{theorem}

Note that the measures $\nu_D$ are $Y^{-\alpha}$-cuspidally tight for any $\alpha \geq 0$ by the compactness of $\Lambda_D$. We may therefore apply \hyperref[thm:noncompact]{Theorem \ref*{thm:noncompact}} in order to prove upper bounds for the $1$-Wasserstein distances in these equidistribution problems.

\begin{theorem}
\label{thm:Duke}
Let $D$ be a fundamental discriminant and let $\nu_D$ denote the Borel probability measure on the modular surface $\Gamma \backslash \Hb$ associated to Heegner points of discriminant $D$ via \eqref{eqn:nuDneq} if $D < 0$ and associated to oriented closed geodesics of discriminant $D$ via \eqref{eqn:nuDpos} if $D > 0$. Then for all $\e > 0$,
\begin{equation}
\label{eqn:W1Duke}
\Wscr_1(\nu_D,\nu) \ll_{\e} |D|^{-\frac{1}{12} + \e}.
\end{equation}
Assuming the generalised Lindel\"{o}f hypothesis, we have the stronger bound
\begin{equation}
\label{eqn:W1DukeGLH}
\Wscr_1(\nu_D,\nu) \ll_{\e} |D|^{-\frac{1}{4} + \e}.
\end{equation}
\end{theorem}

\subsubsection{Mass Equidistribution of Hecke--Maa\ss{} Cusp Forms}

Let $g \in \BB$ be a Hecke--Maa\ss{} cusp form, namely a Maa\ss{} cusp form that is a joint eigenfunction of the Hecke operators $T_n$ for all positive integers $n$. Let $\nu_g$ denote the Borel probability measure given on Borel sets $B \subseteq \Gamma \backslash \Hb$ by
\begin{equation}
\label{eqn:nugdefeq}
\nu_g(B) \coloneqq \int_{B} |g(z)|^2 \, d\mu(z).
\end{equation}
The quantum unique ergodicity conjecture of Rudnick and Sarnak for $\Gamma \backslash \Hb$ predicts (in a stronger form, involving microlocal lifts) that as one traverses a sequence of Hecke--Maa\ss{} cusp forms of increasing spectral parameter $t_g$, the probability measures $\nu_g$ equidistribute on $\Gamma \backslash \Hb$ with respect to $\nu$ \cite[Conjecture]{RS94}. This conjecture was proven by Lindenstrauss \cite[Theorem 1.4]{Lin06} with additional input from Soundararajan \cite{Sou10b}.

\begin{theorem}[Lindenstrauss--Soundararajan {\cite{Lin06,Sou10b}}]
\label{thm:ineffectiveMaassQUE}
Let $(g)$ be a sequence of Hecke--Maa\ss{} cusp forms on $\Gamma \backslash \Hb$ with increasing spectral parameter $t_g$. As $t_g$ tends to infinity, the sequence of Borel probability measures $\nu_g$ associated to $|g|^2$ via \eqref{eqn:nugdefeq} equidistributes on the modular surface with respect to the probability Haar measure $\nu$.
\end{theorem}

The proof of \hyperref[thm:ineffectiveMaassQUE]{Theorem \ref*{thm:ineffectiveMaassQUE}} is via ergodic methods and gives no quantifiable rate of equidistribution. Earlier, Watson proved that the assumption of the generalised Lindel\"{o}f hypothesis (or merely certain as-yet unproven subconvex bounds for various $L$-functions; see \cite{BHMWW24,Nel25}) also implies the quantum unique ergodicity conjecture for $\Gamma \backslash \Hb$ \cite[Corollary 1]{Wat08}. We show how this assumption yields strong upper bounds for the related $1$-Wasserstein distance via an application of \hyperref[thm:noncompact]{Theorem \ref*{thm:noncompact}}, which is applicable as the measures $\nu_g$ are $Y^{-\alpha}$-cuspidally tight for any $\alpha \geq 0$ due to the fact that Maa\ss{} cusp forms decay exponentially at cusps.

\begin{theorem}
\label{thm:MaassQUE}
Let $g \in \BB$ be a Hecke--Maa\ss{} cusp form and let $\nu_g$ denote the Borel probability measure on the modular surface associated to $|g|^2$ via \eqref{eqn:nugdefeq}. Under the assumption of the generalised Lindel\"{o}f hypothesis, we have that for all $\e > 0$,
\[\Wscr_1(\nu_g,\nu) \ll_{\e} t_g^{-\frac{1}{2} + \e}.\]
\end{theorem}

\section{Tools}

In order to insert the weight $e^{-t_f^2/T^2}$ into the inequalities \eqref{eqn:compact} and \eqref{eqn:noncompact}, we must determine the behaviour of the inverse Selberg--Harish-Chandra transform of this function (or, more precisely, a rescaling of this function).

\begin{lemma}
For $T \geq 1$, define $h : \C \to \C$ by
\begin{equation}
\label{eqn:hdefeq}
h(t) \coloneqq e^{-\frac{t^2 + \frac{1}{4}}{2T^2}}.
\end{equation}
Let $k : \R_+ \to \C$ be the inverse Selberg--Harish-Chandra transform (or inverse Mehler--Fock transform) of $h$ given by
\begin{equation}
\label{eqn:kdefeq}
k(u) \coloneqq \frac{1}{4\pi} \int_{-\infty}^{\infty} h(t) P_{-\frac{1}{2} + it}(1 + 2u) t \tanh \pi t \, dt,
\end{equation}
where $P_{\nu}(z)$ denotes the associated Legendre function. Then $k$ is nonnegative and satisfies
\begin{align}
\label{eqn:kint}
\int_{0}^{\infty} k(u) \, du & = \frac{1}{4\pi},	\\
\label{eqn:karsinhint}
\int_{0}^{\infty} k(u) \arsinh \sqrt{u} \, du & \ll \frac{1}{T}.
\end{align}
\end{lemma}

\begin{proof}
Via the Selberg--Harish-Chandra inversion formula (or Mehler--Fock inversion formula), we have that
\[h(t) = 4\pi \int_{0}^{\infty} k(u) P_{-\frac{1}{2} + it}(1 + 2u) \, du\]
(see \cite[Chapter 1]{Iwa02}). Since $P_{-1}(z) = P_0(z) = 1$, we deduce that
\[\int_{0}^{\infty} k(u) \, du = \frac{1}{4\pi} h\left(\frac{i}{2}\right) = \frac{1}{4\pi}.\]

Next, we invoke the identity \cite[(8.715.2) and (8.737.4)]{GR15}
\[P_{-\frac{1}{2} + it}(\cosh \rho) \tanh \pi t = \frac{1}{\pi} \int_{\rho}^{\infty} \frac{\sin(tv)}{\sqrt{\sinh^2 \frac{v}{2} - \sinh^2 \frac{\rho}{2}}} \, dv\]
in order to see that
\[k\left(\sinh^2 \frac{\rho}{2}\right) = \frac{1}{8\pi^2} \int_{\rho}^{\infty} \frac{1}{\sqrt{\sinh^2 \frac{v}{2} - \sinh^2 \frac{\rho}{2}}} \int_{-\infty}^{\infty} h(t) t \sin(tv) \, dt \, dv.\]
For our choice of test function $h$, the inner integral is equal to
\[\sqrt{2\pi} T^3 e^{-\frac{1}{8T^2}} v e^{-\frac{T^2 v^2}{2}},\]
which implies the nonnegativity of $k$.

Finally, by integration by parts, we have that
\begin{align*}
\int_{0}^{\infty} k(u) \arsinh \sqrt{u} \, du & = \frac{1}{4} \int_{0}^{\infty} k\left(\sinh^2 \frac{\rho}{2}\right) \rho \sinh \rho \, d\rho	\\
& = \frac{1}{16\sqrt{2} \pi^{3/2}} T^3 e^{-\frac{1}{8T^2}} \int_{0}^{\infty} \rho \sinh \rho \int_{\rho}^{\infty} \frac{v e^{-\frac{T^2 v^2}{2}}}{\sqrt{\sinh^2 \frac{v}{2} - \sinh^2 \frac{\rho}{2}}} \, dv \, d\rho	\\
& = \frac{1}{4\sqrt{2} \pi^{3/2}} T^3 e^{-\frac{1}{8T^2}} \int_{0}^{\infty} v e^{-\frac{T^2 v^2}{2}} \int_{0}^{v} \sqrt{\sinh^2 \frac{v}{2} - \sinh^2 \frac{\rho}{2}} \, d\rho \, dv.
\end{align*}
The inner integral is bounded by $v \sinh \frac{v}{2}$. Thus this is bounded by
\[\frac{1}{2\pi^{3/2}} e^{-\frac{1}{8T^2}} \int_{0}^{\infty} v^2 e^{-v^2} \sinh \frac{v}{\sqrt{2} T} \, dv.\]
We break up this integral over $[0,\infty)$ into two parts, namely the two ranges $[0,T]$ and $[T,\infty)$. We bound the former integral using the fact that $\sinh \frac{v}{\sqrt{2} T} \ll \frac{v}{T}$ in this range, while we bound the latter integral using the fact that $\sinh \frac{v}{\sqrt{2} T} \ll e^{\frac{v}{\sqrt{2}T}}$ in this range. In this way, we see that both integrals are $O(\frac{1}{T})$.
\end{proof}

\begin{remark}
\label{rem:cutoff}
The holomorphic entire function $h$ given by \eqref{eqn:hdefeq} is nonnegative on $\R$, bounded above by $1$ on $[-T,T]$, decays at a Gaussian rate outside of $[-T,T]$, and is such that its Selberg--Harish-Chandra transform $k$ satisfies $\int_{0}^{\infty} k(u) \, du = \frac{1}{4\pi}$. One might rather work with a nonnegative function on $\R$ that is bounded above by $1$ on $[-T,T]$ and uniformly zero outside this interval. However, for any such function $h$, the integral transform $\int_{-\infty}^{\infty} h(t) t \sin(tv) \, dt$ cannot be $O(e^{-cv})$ for any $c > 0$ (for otherwise $h$ would be holomorphic in a horizontal strip containing $\R$, hence uniformly zero due to the identity theorem, so that its Selberg--Harish-Chandra transform would also be uniformly zero). This means the above method cannot be used to ensure the identity \eqref{eqn:kint} and the bound \eqref{eqn:karsinhint} for the inverse Selberg--Harish-Chandra transform $k$ given by \eqref{eqn:kdefeq}.
\end{remark}

Next, we state some properties of the automorphic kernel associated to $k$.

\begin{lemma}
\label{lem:Kproperties}
Let $\Gamma$ be a lattice in $\SL_2(\R)$. Define the automorphic kernel $K : \Gamma \backslash \Hb \times \Gamma \backslash \Hb \to \C$ by
\begin{equation}
\label{eqn:Kdefeq}
K(z,w) \coloneqq \sum_{\gamma \in \Gamma} k(u(z,\gamma w)),
\end{equation}
where $k$ is as in \eqref{eqn:kdefeq} and $u$ is as in \eqref{eqn:udefeq}. Then $K(w,z) = K(z,w)$ for all $z,w \in \Gamma \backslash \Hb$ and
\begin{equation}
\label{eqn:Kint}
\int_{\Gamma \backslash \Hb} K(z,w) \, d\mu(w) = \int_{\Gamma \backslash \Hb} K(w,z) \, d\mu(w) = 1
\end{equation}
for all $z \in \Gamma \backslash \Hb$.
\end{lemma}

\begin{proof}
The fact that $K(w,z) = K(z,w)$ for all $z,w \in \Gamma \backslash \Hb$ follows from the definition \eqref{eqn:udefeq} of $u$. Next, let $g_0 \in \SL_2(\R)$ be such that $z = g_0 i$. Then by unfolding, making the change of variables $w \mapsto g_0 w$, and using the $\SL_2(\R)$-invariance of $\mu$ and $u$, we have that
\begin{align*}
\int_{\Gamma \backslash \Hb} K(z,w) \, d\mu(w) & = \sum_{\gamma \in \Gamma} \int_{\Gamma \backslash \Hb} k(u(g_0 i,\gamma w)) \, d\mu(w)	\\
& = \int_{\Hb} k(u(g_0 i,w)) \, d\mu(w)	\\
& = \int_{\Hb} k(u(g_0 i, g_0 w)) \, d\mu(w)	\\
& = \int_{\Hb} k(u(i,w)) \, d\mu(w)	\\
& = \int_{0}^{\infty} \int_{-\infty}^{\infty} k\left(\frac{x^2 + (y - 1)^2}{4y^2}\right) \, \frac{dx \, dy}{y^2}.
\end{align*}
We pass to geodesic polar coordinates by setting
\begin{equation}
\label{eqn:geodesicpolar}
x = \frac{2 \sqrt{u(u + 1)} \sin \theta}{1 + 2u + 2 \sqrt{u(u + 1)} \cos \theta}, \qquad y = \frac{1}{1 + 2u + 2 \sqrt{u(u + 1)} \cos \theta},
\end{equation}
so that $u \in \R_+$, $\theta \in [0,2\pi)$, and $y^{-2} \, dx \, dy = 2 \, du \, d\theta$ (see \cite[Section 1.3]{Iwa02}). Thus this double integral becomes
\[2 \int_{0}^{2\pi} \int_{0}^{\infty} k(u) \, du \, d\theta,\]
which is equal to $1$ by \eqref{eqn:kint}.
\end{proof}

The final key tool that we require is a careful approximation of a $1$-Lipschitz function on $\Gamma \backslash \Hb$ by a function in $C^{\infty}(\Gamma \backslash \Hb)$, the space of smooth functions $G : \Hb \to \C$ satisfying $G(\gamma z) = G(z)$ for all $\gamma \in \Gamma$ and $z \in \Hb$.

\begin{lemma}
\label{lem:smoothing}
Let $\Gamma$ be a lattice in $\SL_2(\R)$. Given $F \in \Lip_1(\Gamma \backslash \Hb)$ and $\e > 0$, there exists some $F_{\e} \in C^{\infty}(\Gamma \backslash \Hb)$ for which
\begin{align}
\label{eqn:Fesup}
\sup_{z \in \Gamma \backslash \Hb} |F(z) - F_{\e}(z)| & \leq \e,	\\
\label{eqn:nablaFesup}
\sup_{z \in \Gamma \backslash \Hb} \Im(z)^2 \left|\frac{\dee F_{\e}}{\dee z}\right|^2 & \leq \left(e^{\e} - \frac{1}{2}\right)^2,
\end{align}
where $\frac{\dee}{\dee z} \coloneqq \frac{1}{2} (\frac{\dee}{\dee x} - i \frac{\dee}{\dee y})$ denotes the Wirtinger derivative.
\end{lemma}

\begin{proof}
For fixed $\e > 0$, let $k_{\e} : \R_+ \to \C$ be the smooth nonnegative function
\[k_{\e}(u) \coloneqq \begin{dcases*}
\frac{1}{C\sinh^2 \frac{\e}{2}} \exp\left(\frac{\sinh^2 \frac{\e}{2}}{u^2 - \sinh^2 \frac{\e}{2}}\right) & if $0 < u < \sinh^2 \frac{\e}{2}$,	\\
0 & if $u \geq \sinh^2 \frac{\e}{2}$,
\end{dcases*}\]
where $C \coloneqq 4\pi \int_{0}^{1} \exp(\frac{1}{u^2 - 1}) \, du$. By passing to geodesic polar coordinates, we have that for all $z \in \Hb$,
\begin{equation}
\label{eqn:keint}
\int_{\Hb} k_{\e}(u(z,w)) \, d\mu(w) = 4\pi \int_{0}^{\infty} k_{\e}(u) \, du = 1.
\end{equation}
We let $K_{\e} : \Gamma \backslash \Hb \times \Gamma \backslash \Hb \to \C$ denote the automorphic kernel
\[K_{\e}(z,w) \coloneqq \sum_{\gamma \in \Gamma} k_{\e}(u(z,\gamma w)).\]

We now define the desired function $F_{\e} : \Hb \to \C$ via
\[F_{\e}(z) \coloneqq \int_{\Gamma \backslash \Hb} F(w) K_{\e}(z,w) \, d\mu(w) = \int_{\Hb} F(w) k_{\e}(u(z,w)) \, d\mu(w).\]
The smoothness of $F_{\e}$ is clear from the smoothness of $k_{\e}$, while the $\Gamma$-invariance of $F_{\e}$ follows from the bi-$\Gamma$-invariance of $K_{\e}$.

Next, for any $z \in \Hb$, we have via \eqref{eqn:keint} that
\begin{align*}
\left|F(z) - F_{\e}(z)\right| & = \left|\int_{\Hb} (F(z) - F(w)) k_{\e}(u(z,w)) \, d\mu(w)\right|	\\
& \leq \int_{\Hb} \rho(z,w) k_{\e}(u(z,w)) \, d\mu(w)	\\
& = \int_{\Hb} \rho(i,w) k_{\e}(u(i,w)) \, d\mu(w)
\end{align*}
as $F$ is $1$-Lipschitz and $u,\rho,\mu$ are $\SL_2(\R)$-invariant. Passing to geodesic polar coordinates, observing that $\rho(i,w) = 2\arsinh\sqrt{u(i,w)} \leq \e$ for $u(i,w) \leq \sinh^2 \frac{\e}{2}$, and recalling \eqref{eqn:keint}, we thereby obtain \eqref{eqn:Fesup}.

When $\frac{\dee F_{\e}}{\dee z}$ is nonzero, its modulus is equal to half of the modulus of $\nabla_{v_0} F_{\e}(z)$, the directional derivative of $F_{\e}$ in the direction
\[v_0 = \frac{\frac{\dee F_{\e}}{\dee z}}{\left|\frac{\dee F_{\e}}{\dee z}\right|}.\]
Thus if $\frac{\dee F_{\e}}{\dee z} \neq 0$, then
\[\Im(z)^2 \left|\frac{\dee F_{\e}}{\dee z}\right|^2 = \frac{1}{4} \Im(z)^2 \left|\nabla_{v_0} F_{\e}(z)\right|^2 = \frac{1}{4} \Im(z)^2 \lim_{h \to 0} \left|\frac{F_{\e}(z + hv_0) - F_{\e}(z)}{h}\right|^2.\]
Let $g_0,g_1 \in \SL_2(\R)$ be such that $g_0 i = z$ and $g_1 i = z + hv_0$. Then by making the change of variables $w \mapsto g_0 w$ and $w \mapsto g_1 w$ respectively together with the facts that $u$ and $\mu$ are $\SL_2(\R)$-invariant and $F$ is $1$-Lipschitz, we see that
\begin{align*}
\left|F_{\e}(z + hv_0) - F_{\e}(z)\right| & = \left|\int_{\Hb} F(w) \left(k_{\e}(u(g_1 i,w)) - k_{\e}(u(g_0 i,w))\right) \, d\mu(w)\right|	\\
& = \left|\int_{\Hb} \left(F(g_1 w) - F(g_0 w)\right) k_{\e}(u(i,w)) \, d\mu(w)\right|	\\
& \leq \int_{\Hb} \rho(g_1 w,g_0 w) k_{\e}(u(i,w)) \, d\mu(w).
\end{align*}
Since $g_1 i = z + hv_0$ and $g_0 i = z$, we must have that $g_1 w = \Re(z + hv_0) + \Im(z + hv_0)w$ and $g_0 w = \Re(z) + \Im(z)w$. Thus
\[\rho(g_1 w,g_0 w) = 2 \arsinh \frac{|h| \left|v_0 + \Im(v_0) (w - i)\right|}{2\sqrt{\Im(z)(\Im(z) + h \Im(v_0))} \Im(w)},\]
so that
\[\lim_{h \to 0} \frac{\rho(g_1 w,g_0 w)}{|h|} = \frac{\left|v_0 + \Im(v_0) (w - i)\right|}{\Im(z) \Im(w)} \leq \frac{1 + |w - i|}{\Im(z) \Im(w)}.\]
It follows that
\[\Im(z)^2 \left|\frac{\dee F_{\e}}{\dee z}\right|^2 \leq \frac{1}{4} \left(\int_{\Hb} \frac{1 + |w - i|}{\Im(w)} k_{\e}(u(i,w)) \, d\mu(w)\right)^2.\]
Recall that $k_{\e}(u(i,w))$ vanishes unless $u(i,w) \leq \sinh^2 \frac{\e}{2}$. When this inequality holds, we see via passing once more to geodesic polar coordinates as in \eqref{eqn:geodesicpolar} that
\begin{align*}
\frac{1 + |w - i|}{\Im(w)} & = 1 + 2u + 2 \sqrt{u(u + 1)} \cos \theta + 2\sqrt{u \left(1 + 2u + 2\sqrt{u(u + 1)} \cos \theta\right)}	\\
& \leq 1 + 2 \sinh^2 \frac{\e}{2} + 2 \sinh \frac{\e}{2} \cosh \frac{\e}{2} + 2 \sinh \frac{\e}{2} \sqrt{1 + 2 \sinh^2 \frac{\e}{2} + 2 \sinh \frac{\e}{2} \cosh \frac{\e}{2}}	\\
& = \cosh \e + \sinh \e + 2 \sinh \frac{\e}{2} \sqrt{\cosh \e + \sinh \e}	\\
& = 2e^{\e} - 1.
\end{align*}
With this inequality in hand, the bound \eqref{eqn:nablaFesup} then follows once more from \eqref{eqn:keint}.
\end{proof}

\section{Proofs of \texorpdfstring{\hyperref[thm:compact]{Theorems \ref*{thm:compact}} and \ref{thm:noncompact}}{Theorems \ref{thm:compact} and \ref{thm:noncompact}}}

We proceed to the proofs of \hyperref[thm:compact]{Theorems \ref*{thm:compact}} and \ref{thm:noncompact}. We only give details for the latter, since the proof of the former follows by a similar but simpler argument due to the lack of the continuous spectrum in this setting.

\begin{proof}[Proof of {\hyperref[thm:noncompact]{Theorem \ref*{thm:noncompact}}}]
Fix $\e > 0$. Let $F \in \Lip_1(\Gamma \backslash \Hb)$, let $K$ be as in \eqref{eqn:Kdefeq}, and let $F_{\e}$ be as in \hyperref[lem:smoothing]{Lemma \ref*{lem:smoothing}}. Via \eqref{eqn:Kint} and the triangle inequality, we have that
\begin{equation}
\label{eqn:triangleinequality}
\begin{split}
\hspace{1.5cm} & \hspace{-1.5cm} \left|\int_{\Gamma \backslash \Hb} F(z) \, d\nu_1(z) - \int_{\Gamma \backslash \Hb} F(z) \, d\nu_2(z)\right|	\\
& \leq \left|\int_{\Gamma \backslash \Hb} \int_{\Gamma \backslash \Hb} (F(z) - F(w)) K(z,w) \, d\mu(w) \, d\nu_1(z)\right|	\\
& \quad + \left|\int_{\Gamma \backslash \Hb} \int_{\Gamma \backslash \Hb} (F(z) - F(w)) K(z,w) \, d\mu(w) \, d\nu_2(z)\right|	\\
& \quad + \left|\int_{\Gamma \backslash \Hb} \int_{\Gamma \backslash \Hb} (F(w) - F_{\e}(w)) K(z,w) \, d\mu(w) \, d\nu_1(z)\right|	\\
& \quad + \left|\int_{\Gamma \backslash \Hb} \int_{\Gamma \backslash \Hb} (F(w) - F_{\e}(w)) K(z,w) \, d\mu(w) \, d\nu_2(z)\right|	\\
& \quad + \left|\int_{\Gamma \backslash \Hb} \int_{\Gamma \backslash \Hb} F_{\e}(w) K(z,w) \, d\mu(w) \, d\nu_1(z) - \int_{\Gamma \backslash \Hb} \int_{\Gamma \backslash \Hb} F_{\e}(w) K(z,w) \, d\mu(w) \, d\nu_2(z)\right|.
\end{split}
\end{equation}

To bound the first term on the right-hand side of \eqref{eqn:triangleinequality}, we use the triangle inequality and unfold the inner integral, so that this is bounded by
\[\int_{\Gamma \backslash \Hb} \int_{\Hb} |F(z) - F(w)| |k(u(z,w))| \, d\mu(w) \, d\nu_1(z).\]
Since $F$ is $1$-Lipschitz, the inner integral is in turn bounded by
\[\int_{\Hb} |k(u(z,w))| \rho(z,w) \, d\mu(w).\]
Via the $\SL_2(\R)$-invariance of $u,\rho,d\mu$ and then passing to geodesic polar coordinates, this is equal to
\[8\pi \int_{0}^{\infty} |k(u)| \arsinh \sqrt{u} \, du.\]
By the nonnegativity of $k$, the bound \eqref{eqn:karsinhint}, and the fact that $\nu_1$ is a probability measure, we deduce that the first term on the right-hand side of \eqref{eqn:triangleinequality} is $O(\frac{1}{T})$ independently of $F \in \Lip_1(\Gamma \backslash \Hb)$. The same argument yields the same bound for the second term.

To bound the third term on the right-hand side of \eqref{eqn:triangleinequality}, we use the triangle inequality and unfold the inner integral, so that this is bounded by
\[\int_{\Gamma \backslash \Hb} \int_{\Hb} |F(w) - F_{\e}(w)| |k(u(z,w))| \, d\mu(w) \, d\nu_1(z).\]
Via the bound \eqref{eqn:Fesup} for $|F(w) - F_{\e}(w)|$, the $\SL_2(\R)$-invariance of $d\mu$, the nonnegativity of $k$, the bound \eqref{eqn:kint}, and the fact that $\nu_1$ is a probability measure, this is at most $\e$. The same argument yields the same bound for the second term.

We are left with bounding the fifth term on the right-hand side of \eqref{eqn:triangleinequality}. We do so by invoking Parseval's identity for $L^2(\Gamma \backslash \Hb)$ \cite[Theorem 7.3]{Iwa02}; the ensuing expression involves a sum over an orthonormal basis $\BB$ of Hecke--Maa\ss{} cusp forms $f$ and a finite sum over cusps $\af$ of integrals indexed by Eisenstein series $E_{\af}(z,s)$, as discussed previously in \hyperref[sec:noncompact]{Section \ref*{sec:noncompact}}. Together with \eqref{eqn:Kint}, we deduce that
\begin{multline}
\label{eqn:Parsevalstep}
\int_{\Gamma \backslash \Hb} F_{\e}(w) K(z,w) \, d\mu(w) = \frac{\langle F_{\e},1\rangle}{\mu(\Gamma \backslash \Hb)} + \sum_{f \in \BB} \langle F_{\e},f\rangle \int_{\Gamma \backslash \Hb} f(w) K(z,w) \, d\mu(w)	\\
+ \sum_{\af} \frac{1}{4\pi} \int_{-\infty}^{\infty} \left\langle F_{\e}, E_{\af}\left(\cdot, \frac{1}{2} + it\right)\right\rangle \int_{\Gamma \backslash \Hb} E_{\af}\left(w,\frac{1}{2} + it\right) K(z,w) \, d\mu(w) \, dt,
\end{multline}
where $\langle \cdot,\cdot,\rangle$ denotes the Petersson inner product on $L^2(\Gamma \backslash \Hb)$ given by
\[\langle F_1,F_2\rangle \coloneqq \int_{\Gamma \backslash \Hb} F_1(z) \overline{F_2(z)} \, d\mu(z).\]
Since $f$ is a Laplacian eigenfunction with Laplacian eigenvalue $\frac{1}{4} + t_f^2$, and similarly $E_{\af}(w,\frac{1}{2} + it)$ is a Laplacian eigenfunction with Laplacian eigenvalue $\frac{1}{4} + t^2$, we have that
\begin{align*}
\int_{\Gamma \backslash \Hb} f(w) K(z,w) \, d\mu(z) & = \int_{\Hb} f(w) k(u(z,w)) \, d\mu(z)	\\
& = h(t_f) f(z),	\\
\int_{\Gamma \backslash \Hb} E_{\af}\left(w,\frac{1}{2} + it\right) K(z,w) \, d\mu(w) & = \int_{\Hb} E_{\af}\left(w,\frac{1}{2} + it\right) k(u(z,w)) \, d\mu(w)	\\
& = h(t) E_{\af}\left(z,\frac{1}{2} + it\right),
\end{align*}
where $h$ is as in \eqref{eqn:hdefeq} \cite[Theorem 1.14]{Iwa02}. From this, we see that the right-hand side of \eqref{eqn:Parsevalstep} may be rewritten as
\[\frac{\langle F_{\e},1\rangle}{\mu(\Gamma \backslash \Hb)} + \sum_{f \in \BB} e^{-\frac{t_f^2 + \frac{1}{4}}{2T^2}} \langle F_{\e},f\rangle f(z) + \sum_{\af} \frac{1}{4\pi} \int_{-\infty}^{\infty} e^{-\frac{t^2 + \frac{1}{4}}{2T^2}} \left\langle F_{\e}, E_{\af}\left(\cdot, \frac{1}{2} + it\right)\right\rangle E_{\af}\left(z,\frac{1}{2} + it\right) \, dt.\]
In this way, we naturally introduce the weight \eqref{eqn:hdefeq}. We now integrate both sides over $z \in \Gamma \backslash \Hb$ with respect to $\nu_1$ and with respect to $\nu_2$ and take the difference. It is at this point that we require the $Y^{-1/2 - \delta}$-cuspidally tightness of $\nu_1$ and $\nu_2$ in order to ensure that each Eisenstein series $E_{\af}(\cdot,\frac{1}{2} + it)$ is $\nu_1$- and $\nu_2$-integrable. Indeed, we may interchange the order of integration as absolute convergence is guaranteed via the Cauchy--Schwarz inequality together with the local Weyl law \cite[Proposition 7.2]{Iwa02}, which states that for $z \in \Hb$ and $U \geq 1$,
\[\sum_{\substack{f \in \BB \\ |t_f| \leq U}} |f(z)|^2 + \sum_{\af} \frac{1}{4\pi} \int_{-U}^{U} \left|E_{\af}\left(z,\frac{1}{2} + it\right)\right|^2 \, dt \ll_{\Gamma} U^2 + U \Ht_{\Gamma}(z).\]
We deduce via the Cauchy--Schwarz inequality that
\begin{multline}
\label{eqn:CauchySchwarz}
\left|\int_{\Gamma \backslash \Hb} \int_{\Gamma \backslash \Hb} F_{\e}(w) K(z,w) \, d\mu(w) \, d\nu_1(z) - \int_{\Gamma \backslash \Hb} \int_{\Gamma \backslash \Hb} F_{\e}(w) K(z,w) \, d\mu(w) \, d\nu_2(z)\right|^2	\\
\leq \left(\sum_{f \in \BB} \left(\frac{1}{4} + t_f^2\right) |\langle F_{\e},f\rangle|^2 + \sum_{\af} \frac{1}{4\pi} \int_{-\infty}^{\infty} \left(\frac{1}{4} + t^2\right) \left|\left\langle F_{\e}, E_{\af}\left(\cdot, \frac{1}{2} + it\right)\right\rangle\right|^2 \, dt\right)	\\
\times \left(\sum_{f \in \BB} \frac{e^{-\frac{t_f^2 + \frac{1}{4}}{T^2}}}{\frac{1}{4} + t_f^2} \left|\int_{\Gamma \backslash \Hb} f(z) \, d\nu_1(z) - \int_{\Gamma \backslash \Hb} f(z) \, d\nu_2(z)\right|^2 \right.	\\
\left. + \sum_{\af} \frac{1}{4\pi} \int_{-\infty}^{\infty} \frac{e^{-\frac{t^2 + \frac{1}{4}}{T^2}}}{\frac{1}{4} + t^2} \left|\int_{\Gamma \backslash \Hb} E_{\af}\left(z,\frac{1}{2} + it\right) \, d\nu_1(z) - \int_{\Gamma \backslash \Hb} E_{\af}\left(z,\frac{1}{2} + it\right) \, d\nu_2(z)\right|^2 \, dt\right).
\end{multline}
To bound the second line of \eqref{eqn:CauchySchwarz}, we use the self-adjointness of $\Delta$ with respect to the Petersson inner product, Parseval's identity, and Green's first identity (see \cite[Lemma 4.1]{Iwa02}) in order to see that
\begin{align*}
\hspace{1.5cm} & \hspace{-1.5cm} \sum_{f \in \BB} \left(\frac{1}{4} + t_f^2\right) |\langle F_{\e},f\rangle|^2 + \sum_{\af} \frac{1}{4\pi} \int_{-\infty}^{\infty} \left(\frac{1}{4} + t^2\right) \left|\left\langle F_{\e}, E_{\af}\left(\cdot, \frac{1}{2} + it\right)\right\rangle\right|^2 \, dt	\\
& = \sum_{f \in \BB} \langle \Delta F_{\e},f\rangle \langle f,F_{\e}\rangle + \sum_{\af} \frac{1}{4\pi} \int_{-\infty}^{\infty} \left\langle \Delta F_{\e}, E_{\af}\left(\cdot, \frac{1}{2} + it\right)\right\rangle \left\langle E_{\af}\left(\cdot, \frac{1}{2} + it\right), F_{\e}\right\rangle \, dt	\\
& = \langle \Delta F_{\e},F_{\e}\rangle	\\
& = 4 \int_{\Gamma \backslash \Hb} \Im(z)^2 \left|\frac{\dee F_{\e}}{\dee z}\right|^2 \, d\mu(z).
\end{align*}
By \eqref{eqn:nablaFesup}, this is at most $\mu(\Gamma \backslash \Hb) (2e^{\e} - 1)^2$.

We have therefore shown that for $F \in \Lip_1(\Gamma \backslash \Hb)$ and $\e > 0$,
\begin{multline*}
\left|\int_{\Gamma \backslash \Hb} F(z) \, d\nu_1(z) - \int_{\Gamma \backslash \Hb} F(z) \, d\nu_2(z)\right|	\\
\ll \frac{1}{T} + 2\e + \mu(\Gamma \backslash \Hb)^{\frac{1}{2}} (2e^{\e} - 1) \left(\sum_{f \in \BB} \frac{e^{-\frac{t_f^2}{T^2}}}{\frac{1}{4} + t_f^2} \left|\int_{\Gamma \backslash \Hb} f(z) \, d\nu_1(z) - \int_{\Gamma \backslash \Hb} f(z) \, d\nu_2(z)\right|^2 \right.	\\
\left. + \sum_{\af} \frac{1}{4\pi} \int_{-\infty}^{\infty} \frac{e^{-\frac{t^2}{T^2}}}{\frac{1}{4} + t^2} \left|\int_{\Gamma \backslash \Hb} E_{\af}\left(z,\frac{1}{2} + it\right) \, d\nu_1(z) - \int_{\Gamma \backslash \Hb} E_{\af}\left(z,\frac{1}{2} + it\right) \, d\nu_2(z)\right|^2 \, dt\right)^{\frac{1}{2}}.
\end{multline*}
Since $F \in \Lip_1(\Gamma \backslash \Hb)$ and $\e > 0$ were arbitrary, \eqref{eqn:compact} now follows via the Kantorovich--Rubinstein duality theorem \eqref{eqn:KRduality}.
\end{proof}

\section{Arithmetic Applications}

\subsection{Proof of \texorpdfstring{\hyperref[thm:Duke]{Theorem \ref*{thm:Duke}}}{Theorem \ref{thm:Duke}}}

We now prove \hyperref[thm:Duke]{Theorem \ref*{thm:Duke}} via the Berry--Esseen inequality given in \hyperref[thm:noncompact]{Theorem \ref*{thm:noncompact}} in the special case that $\Gamma = \SL_2(\Z)$. The chief inputs are deep results due to Waldspuger that give exact formul\ae{} for each of the Weyl sums $\int_{\Gamma \backslash \Hb} f(z) \, d\nu_D(z)$ and $\int_{\Gamma \backslash \Hb} E(z,\frac{1}{2} + it) \, d\nu_D(z)$ in terms of $L$-functions together with bounds for mixed moments of $L$-functions. Throughout, we take $\Gamma = \SL_2(\Z)$ and choose the orthonormal basis $\BB$ of Maa\ss{} cusp forms on $\Gamma \backslash \Hb$ to consist of Hecke--Maa\ss{} cusp forms, which are joint eigenfunctions of the Hecke operators $T_n$ for all positive integers $n$.

\begin{lemma}[{Waldspurger \cite{Wal85} (see \cite[Theorems 3 and 5 and (5.17)]{DIT16})}]
\label{lem:DukeWaldspurger}
For a fundamental discriminant $D$ and for $f \in \BB$ and $t \in \R$, we have that
\begin{align*}
\left|\int_{\Gamma \backslash \Hb} f(z) \, \nu_D(z)\right|^2 & = \frac{H_{\sgn(D)}(t_f)}{8\sqrt{|D|} L(1,\chi_D)^2} \frac{L\left(\frac{1}{2},f\right) L\left(\frac{1}{2}, f \otimes \chi_D\right)}{L(1,\ad f)},	\\
\left|\int_{\Gamma \backslash \Hb} E\left(z,\frac{1}{2} + it\right) \, d\nu_D(z)\right|^2 & = \frac{H_{\sgn(D)}(t)}{4\sqrt{|D|} L(1,\chi_D)^2} \left|\frac{\zeta\left(\frac{1}{2} + it\right) L\left(\frac{1}{2} + it, \chi_D\right)}{\zeta(1 + 2it)}\right|^2,
\end{align*}
where $\chi_D$ denotes the primitive quadratic Dirichlet character modulo $|D|$ and
\begin{align*}
H_{-}(t) & \coloneqq 2\pi^2,	\\
H_{+}(t) & \coloneqq \frac{\Gamma\left(\frac{1}{4} + \frac{it}{2}\right)^2 \Gamma\left(\frac{1}{4} - \frac{it}{2}\right)^2}{\Gamma\left(\frac{1}{2} + it\right) \Gamma\left(\frac{1}{2} - it\right)}.
\end{align*}
\end{lemma}

\begin{proof}[Proof of {\hyperref[thm:Duke]{Theorem \ref*{thm:Duke}}}]
We apply \hyperref[thm:noncompact]{Theorem \ref*{thm:noncompact}} with $\Gamma = \SL_2(\Z)$, $\nu_1 = \nu_D$, $\nu_2 = \nu$, and $T = D^{1/12}$. By \hyperref[lem:DukeWaldspurger]{Lemma \ref*{lem:DukeWaldspurger}}, in order to prove the bound \eqref{eqn:W1Duke}, we must show that
\begin{multline}
\label{eqn:Dukemoment}
\frac{1}{8\sqrt{|D|} L(1,\chi_D)^2} \sum_{f \in \BB} \frac{H_{\sgn(D)}(t_f) e^{-\frac{t_f^2}{T^2}}}{\frac{1}{4} + t_f^2} \frac{L\left(\frac{1}{2},f\right) L\left(\frac{1}{2}, f \otimes \chi_D\right)}{L(1,\ad f)}	\\
+ \frac{1}{16\pi \sqrt{|D|} L(1,\chi_D)^2} \int_{-\infty}^{\infty} \frac{H_{\sgn(D)}(t) e^{-\frac{t^2}{T^2}}}{\frac{1}{4} + t^2} \left|\frac{\zeta\left(\frac{1}{2} + it\right) L\left(\frac{1}{2} + it, \chi_D\right)}{\zeta(1 + 2it)}\right|^2 \, dt \ll_{\e} |D|^{-\frac{1}{6} + \e}.
\end{multline}
By Stirling's formula, we have that
\[H_{-}(t) \ll 1, \qquad H_{+}(t) \ll \frac{1}{1 + |t|}.\]
After inputting Siegel's (ineffective) lower bound $L(1,\chi_D) \gg_{\e} |D|^{-\e}$, the bound \eqref{eqn:Dukemoment} follows via a dyadic subdivision of the sum over $f \in \BB$ and integral over $t \in \R$, H\"{o}lder's inequality with exponents $(\frac{1}{3},\frac{1}{2},\frac{1}{6})$, and the bounds, for $U \geq 1$,
\begin{align}
\label{eqn:weightedWeyl}
\sum_{\substack{f \in \BB \\ t_f \leq U}} \frac{1}{L(1,\ad f)} + \frac{1}{2\pi} \int_{-U}^{U} \left|\frac{1}{\zeta(1 + 2it)}\right|^2 \, dt & \ll U^2,	\\
\notag
\sum_{\substack{f \in \BB \\ t_f \leq U}} \frac{L\left(\frac{1}{2},f\right)^2}{L(1,\ad f)} + \frac{1}{2\pi} \int_{-U}^{U} \left|\frac{\zeta\left(\frac{1}{2} + it\right)}{\zeta(1 + 2it)}\right|^2 \, dt & \ll_{\e} U^{2 + \e}	\\
\notag
\sum_{\substack{f \in \BB \\ t_f \leq U}} \frac{L\left(\frac{1}{2},f \otimes \chi_D\right)^3}{L(1,\ad f)} + \frac{1}{2\pi} \int_{-U}^{U} \left|\frac{L\left(\frac{1}{2} + it,\chi_D\right)^3}{\zeta(1 + 2it)}\right|^2 \, dt & \ll_{\e} (|D| U)^{2 + \e}.
\end{align}
Here the first bound is simply the weighted Weyl law, which is a straightforward application of the Kuznetsov formula \cite[Theorem 9.3]{Iwa02} with $m = n = 1$ and $h$ as in \eqref{eqn:hdefeq} (see, for example, \cite[Lemma 2.4]{Mot97} or \cite[(16.56)]{IK04}), while the second bound is a standard consequence of the approximate functional equation \cite[Theorem 5.3]{IK04} and the spectral large sieve, as in \cite[Theorem 3.3]{Mot97} or \cite[Theorem 7.24]{IK04}. The third bound is due to Andersen and Wu \cite[Theorem 4.1]{AW23} (see also \cite[Theorem 11.1]{GHLN26}), building on earlier work of Conrey and Iwaniec \cite{CI00} and Young \cite{You17}. Note additionally that $L(\frac{1}{2},f \otimes \chi_D)$ is known to be nonnegative via the work of Waldspurger \cite[Th\'{e}or\`{e}me 1]{Wal81}.

Assuming the generalised Lindel\"{o}f hypothesis, we may instead take $T = |D|^{-1/4}$. We again input Siegel's lower bound and perform a dyadic subdivision, but then instead use the pointwise bounds
\begin{align*}
L\left(\frac{1}{2},f\right) L\left(\frac{1}{2}, f \otimes \chi_D\right) & \ll_{\e} (t_f |D|)^{\e},	\\
\left|\zeta\left(\frac{1}{2} + it\right) L\left(\frac{1}{2} + it,\chi_D\right)\right|^2 & \ll_{\e} ((1 + |t|) |D|)^{\e}
\end{align*}
that follow from the assumption of the generalised Lindel\"{o}f hypothesis, followed by the weighted Weyl law \eqref{eqn:weightedWeyl}. This yields the desired bound \eqref{eqn:W1DukeGLH}.
\end{proof}

\begin{remark}
In \cite[Proposition 2.14]{HR22}, it is shown that if $D$ is a squarefree fundamental discriminant, then
\begin{multline*}
\sum_{\substack{f \in \BB \\ U \leq t_f \leq 2U}} \frac{L\left(\frac{1}{2},f\right) L\left(\frac{1}{2},f \otimes \chi_D\right)}{L(1,\ad f)} + \frac{1}{2\pi} \int\limits_{U \leq |t| \leq 2U} \left|\frac{\zeta\left(\frac{1}{2} + it\right) L\left(\frac{1}{2} + it,\chi_D\right)}{\zeta(1 + 2it)}\right|^2 \, dt	\\
\ll_{\e} \begin{dcases*}
|D|^{\frac{1}{3} + \e} U^{2 + \e} & for $U \ll |D|^{\frac{1}{12}}$,	\\
|D|^{\frac{1}{2} + \e} & for $|D|^{\frac{1}{12}} \ll U \ll |D|^{\frac{1}{4}}$,	\\
|D|^{\e} U^{2 + \e} & for $U \gg |D|^{\frac{1}{4}}$.
\end{dcases*}
\end{multline*}
This yields the same bounds as those obtained via H\"{o}lder's inequality for $U \ll |D|^{1/12}$ and stronger bounds for $U \gg |D|^{1/12}$, but does not improve the bound \eqref{eqn:W1Duke}.
\end{remark}

\begin{remark}
There are different ways to quantify the rate of equidistribution in Duke's theorem other than bounds for the $1$-Wasserstein distance $\Wscr_1(\nu_D,\nu)$. For example, one can prove a variant of Duke's theorem involving shrinking targets, where one seeks to show that $\frac{\nu_D(B_D)}{\nu(B_D)} = 1 + o(1)$ for a sequence of sets $B_D$ of shrinking area, such as balls $B_R(z)$ whose radius $R$ shrinks as $|D|$ grows. Young \cite[Theorem 2.1]{You17} has proven a power-saving rate of equidistribution for this shrinking target problem (see also \cite[Theorem 1.24]{Hum18}). One can similarly prove bounds for the ball discrepancy
\[\sup_{B_R(z) \subset \Gamma \backslash \Hb} \left|\nu_D(B_R(y)) - \nu(B_R(z))\right|.\]
Finally, one can study the $L^2$-shrinking target problem, namely bounds (or even asymptotic formul\ae{}) for the variance
\[\int_{\Gamma \backslash \Hb} \left|\nu_D(B_R(z)) - \nu(B_R(z))\right|^2 \, d\nu(z)\]
with $R$ shrinking as $|D|$ grows; see, in particular, \cite{Fav22,Hum18,HR22}.
\end{remark}

\subsection{Proof of \texorpdfstring{\hyperref[thm:MaassQUE]{Theorem \ref*{thm:MaassQUE}}}{Theorem \ref{thm:MaassQUE}}}

Next, we prove \hyperref[thm:MaassQUE]{Theorem \ref*{thm:MaassQUE}} via the Berry--Esseen inequality given in \hyperref[thm:noncompact]{Theorem \ref*{thm:noncompact}}. Once more, the chief input is deep results due to Watson and Ichino that give exact formul\ae{} for the Weyl sums $\int_{\Gamma \backslash \Hb} f(z) \, d\nu_g(z)$ and $\int_{\Gamma \backslash \Hb} E(z,\frac{1}{2} + it) \, d\nu_g(z)$ in terms of $L$-functions together with bounds for mixed moments of $L$-functions.

\begin{lemma}[{Watson \cite{Wat08}, Ichino \cite{Ich08} (see \cite[Proposition 2.8]{Hum18})}]
\label{lem:WatsonIchino}
For $f,g \in \BB$ and $t \in \R$, we have that
\begin{align*}
\left|\int_{\Gamma \backslash \Hb} f(z) \, d\nu_g(z)\right|^2 & = \frac{\pi H(t_f,t_g)}{8 L(1,\ad g)^2} \frac{L\left(\frac{1}{2},f\right) L\left(\frac{1}{2},\ad g \otimes f\right)}{L(1,\ad f)},	\\
\left|\int_{\Gamma \backslash \Hb} E\left(z,\frac{1}{2} + it\right) \, d\nu_g(z)\right|^2 & = \frac{\pi H(t,t_g)}{4 L(1,\ad g)^2} \left|\frac{\zeta\left(\frac{1}{2} + it\right) L\left(\frac{1}{2} + it,\ad g\right)}{\zeta(1 + 2it)}\right|^2,
\end{align*}
where
\begin{multline*}
H(t,t_g) \coloneqq \frac{\Gamma\left(\frac{1}{4} + \frac{it}{2}\right)^2 \Gamma\left(\frac{1}{4} - \frac{it}{2}\right)^2}{\Gamma\left(\frac{1}{2} + it\right) \Gamma\left(\frac{1}{2} - it\right)}	\\
\times \frac{\Gamma\left(\frac{1}{4} + \frac{i(2t_g + t)}{2}\right) \Gamma\left(\frac{1}{4} - \frac{i(2t_g + t)}{2}\right) \Gamma\left(\frac{1}{4} + \frac{i(2t_g - t)}{2}\right) \Gamma\left(\frac{1}{4} - \frac{i(2t_g - t)}{2}\right)}{\Gamma\left(\frac{1}{2} + it_g\right)^2 \Gamma\left(\frac{1}{2} - it_g\right)^2}.
\end{multline*}
\end{lemma}

\begin{proof}[Proof of {\hyperref[thm:MaassQUE]{Theorem \ref*{thm:MaassQUE}}}]
We apply \hyperref[thm:noncompact]{Theorem \ref*{thm:noncompact}} with $\Gamma = \SL_2(\Z)$, $\nu_1 = \nu_g$, $\nu_2 = \nu$, and $T = t_g^{1/2}$. By \hyperref[lem:WatsonIchino]{Lemma \ref*{lem:WatsonIchino}}, we must show that
\begin{multline}
\label{eqn:MaassQUEmoment}
\frac{\pi}{8 L(1,\ad g)^2} \sum_{f \in \BB} \frac{H(t_f,t_g) e^{-\frac{t_f^2}{T^2}}}{\frac{1}{4} + t_f^2} \frac{L\left(\frac{1}{2},f\right) L\left(\frac{1}{2},\ad g \otimes f\right)}{L(1,\ad f)}	\\
+ \frac{1}{16 L(1,\ad g)^2} \int_{-\infty}^{\infty} \frac{H(t,t_g) e^{-\frac{t^2}{T^2}}}{\frac{1}{4} + t^2} \left|\frac{\zeta\left(\frac{1}{2} + it\right) L\left(\frac{1}{2} + it,\ad g\right)}{\zeta(1 + 2it)}\right|^2 \, dt \ll_{\e} t_g^{-1 + \e}.
\end{multline}
By Stirling's formula, we have that
\[H(t,t_g) \ll \begin{dcases*}
\frac{1}{(1 + |t|) t_g} & if $|t| \leq t_g$,	\\
\frac{1}{(1 + 2t_g - |t|)^{1/2} t_g^{3/2}} & if $t_g \leq |t| \leq 2t_g$,	\\
\frac{e^{-\pi(|t| - 2t_g)}}{(1 + |t| - 2t_g)^{1/2} |t|^{3/2}} & if $2t_g \leq |t| \leq 3t_g$,	\\
\frac{e^{-\pi(|t| - 2t_g)}}{|t|^2} & if $|t| \geq 3t_g$.
\end{dcases*}\]
After inputting the Hoffstein--Lockhart lower bound $L(1,\ad g) \gg 1/\log t_g$ \cite{HL94}, the bound \eqref{eqn:MaassQUEmoment} follows via a dyadic subdivision, the pointwise bounds
\begin{align*}
L\left(\frac{1}{2},f\right) L\left(\frac{1}{2},\ad g \otimes f\right) & \ll_{\e} (t_f t_g)^{\e},	\\
\left|\zeta\left(\frac{1}{2} + it\right) L\left(\frac{1}{2} + it,\ad g\right)\right|^2 & \ll_{\e} ((1 + |t|) t_g)^{\e}
\end{align*}
that follow from the assumption of the generalised Lindel\"{o}f hypothesis, and the weighted Weyl law \eqref{eqn:weightedWeyl}.
\end{proof}

\begin{remark}
Just as for Duke's theorem, there are other methods to quantify the rate of equidistribution of mass of Hecke--Maa\ss{} cusp forms. Under the assumption of the generalised Lindel\"{o}f hypothesis, Young has proven small scale mass equidistribution in balls $B_R(z)$ whose radius $R$ shrinks at any rate slightly larger than $t_g^{-1/3}$ \cite[Proposition 1.5]{You16}. Young's result also gives a conditional resolution of a conjecture of Lou and Sarnak \cite[p.\ 210]{LS95} on the size of the ball discrepancy for this equidistribution problem, namely the bound
\[\sup_{B_R(z) \subset \Gamma \backslash \Hb} \left|\nu_g(B_R(y)) - \nu(B_R(z))\right| \ll_{\e} t_g^{-\frac{1}{2} + \e}.\]
Finally, one can study the $L^2$-shrinking target problem in this setting, namely bounds for the variance
\[\int_{\Gamma \backslash \Hb} \left|\nu_g(B_R(z)) - \nu(B_R(z))\right|^2 \, d\nu(z)\]
with $R$ shrinking as $t_g$ grows; in particular, under the assumption of the generalised Lindel\"{o}f hypothesis, one obtains equidistribution in almost every shrinking ball whose radius shrinks at any rate slightly larger than the Planck scale $t_g^{-1}$ \cite[Theorem 1.17]{Hum18}.
\end{remark}

\begin{remark}
There are several other variants of mass equidistribution of cusp forms that one can study. For example, one can prove an analogous variant of \hyperref[thm:MaassQUE]{Theorem \ref{thm:MaassQUE}}, again conditional on the generalised Lindel\"{o}f hypothesis, for the mass equidistribution in the \emph{weight aspect} of holomorphic Hecke cusp forms of increasing weight. This is due to the fact that the Watson--Ichino triple product formula again expresses the relevant Weyl sums in terms of $L$-functions, just as in \hyperref[lem:WatsonIchino]{Lemma \ref*{lem:WatsonIchino}}, and the generalised Lindel\"{o}f hypothesis bounds these essentially optimally.

In a different direction, one can prove \emph{unconditionally} a variant of \hyperref[thm:MaassQUE]{Theorem \ref{thm:MaassQUE}} for the mass equidistribution in the \emph{depth aspect} of Hecke--Maa\ss{} newforms of bounded spectral parameter and increasing prime power level $p^n$ with $p$ fixed and $n$ growing (or alternatively holomorphic Hecke newforms of bounded weight and increasing prime power level). This is due to work of Nelson, Pitale, and Saha, who prove unconditional power-saving bounds for the Weyl sums for this equidistribution problem \cite[Proposition 3.4]{NPS14}.

Mass equidistribution is also known \emph{unconditionally} for holomorphic Hecke cusp forms of increasing weight due to Holowinsky and Soundararajan \cite{Hol10,HS10,Sou10a} and for holomorphic Hecke newforms of increasing weight or arbitrary level (not necessarily a prime power) due to Nelson \cite{Nel11} and Nelson, Pitale, and Saha \cite{NPS14}. These results rely on a different spectral expansion on $L^2(\Gamma \backslash \Hb)$ involving \emph{incomplete} Eisenstein series. The treatment of the Weyl sums involving incomplete Eisenstein series, via sieve theory (see \cite{Hol10}), does not seem to apply directly to Eisenstein series. Since \hyperref[thm:noncompact]{Theorem \ref*{thm:noncompact}} involves Eisenstein series rather than incomplete Eisenstein series, this obstacle prevents \hyperref[thm:noncompact]{Theorem \ref*{thm:noncompact}} from being applicable to proving bounds for the $1$-Wasserstein distance in these equidistribution problems.
\end{remark}

\phantomsection
\addcontentsline{toc}{section}{Acknowledgements}
\hypersetup{bookmarksdepth=-1}

\subsection*{Acknowledgements}

I would like to thank Emmanuel Kowalski and Th\'{e}o Untrau for useful discussions and the referee for their careful reading of this paper.

\hypersetup{bookmarksdepth}

\end{document}